\documentclass[12pt,reqno]{amsart}
\usepackage[utf8]{inputenc}

\usepackage{amsfonts, amssymb, amsmath, amsthm, mathtools, thmtools}
\usepackage{mathrsfs} 
\usepackage{latexsym}
\usepackage{enumerate}
\usepackage{multicol}
\usepackage{times}
\usepackage{verbatim}
\usepackage{tikz-cd}
\usepackage{fullpage}
\usepackage{here}
\usepackage{url}
\usepackage{csquotes}
\usepackage{placeins}
\usepackage{epic}
\usepackage{graphicx}
\usepackage{epstopdf}
\usepackage{pstricks}
\usepackage{pst-plot}
\usepackage{tikz}
\usepackage{xcolor}
\usepackage{subcaption}

\usetikzlibrary{shapes.geometric, positioning, calc}

\input xypic
\xyoption{all}
\xyoption{poly}

\usepackage[colorlinks=true]{hyperref}
\hypersetup{citecolor=blue, linkcolor=blue}

\usepackage[capitalise]{cleveref}

\crefname{equation}{equation}{equations}
\crefname{figure}{{\sc Figure}}{{\sc Figure}}
\crefname{subsection}{Subsection}{Subsections}

\usetikzlibrary{matrix,arrows,decorations.pathmorphing}

\usepackage[euler]{textgreek}
\usepackage{tikz,tikz-cd}
\usepackage[colorinlistoftodos, textwidth = 2.3cm]{todonotes}

\newtheorem{theorem}{Theorem}[section]
\newtheorem{proposition}[theorem]{Proposition}
\newtheorem{lemma}[theorem]{Lemma}

\newtheorem{conjecture}[theorem]{Conjecture}

\newtheorem*{claim*}{Claim}

\theoremstyle{definition}

\newtheorem{remark}[theorem]{Remark}

\newcommand{\Fp}{\mathbb F_p}
\newcommand{\Fps}{\mathbb F_p^*}
\newcommand{\rsum}{\mathbin{\widehat{+}}}
\newcommand{\E}{\mathsf E}
\newcommand{\F}{\mathbb F}

\title{Multiplicative subgroups are not restricted sumsets}
\author{Chi Hoi Yip}
\address{Department of Mathematics, Hong Kong University of Science and Technology, Clear Water Bay, Hong Kong}
\email{machyip@ust.hk}
\author{Semin Yoo}
\address{Discrete Mathematics Group \\ Institute for Basic Science \\ 55 Expo-ro Yuseong-gu, Daejeon 34126 \\ South Korea}
\email{syoo19@ibs.re.kr}
\begin{document}

\begin{abstract}
We determine exactly which proper multiplicative subgroups of a prime field can be represented as a restricted sumset of the form $A\rsum A=\{a+a':a,a'\in A,\ a\ne a'\}$.  We prove that a proper multiplicative subgroup $H\le\F_p^*$ cannot satisfy $H=A\rsum A$ whenever $|H|\ge7$, and that this threshold is sharp.  In fact, such a decomposition exists precisely when $|H|\in\{1,3,6\}$, and we classify all decompositions in these exceptional cases.  This gives a sharp, complete resolution of the restricted-sumset analogue of the generalized S\'ark\"ozy conjecture over prime fields.  This significantly extends and refines previous results of Shkredov and Yip. 
\end{abstract}

\maketitle

\section{Introduction}

Throughout the paper, let $p$ be a prime and let $\F_p^*=\F_p\setminus\{0\}$.  For a divisor $d>1$ of $p-1$, write
\[
        S_d=\{x^d:x\in\F_p^*\},
        \qquad
        M=|S_d|=\frac{p-1}{d}.
\]
Thus every proper multiplicative subgroup of $\F_p^*$ is of the form $S_d$.

A long-standing theme in arithmetic combinatorics is to determine whether a set with strong multiplicative structure can admit a nontrivial additive decomposition.  A celebrated conjecture of S\'ark\"ozy~\cite{S12} asserted that, for every sufficiently large prime $p$, the subgroup $S_2$ of nonzero quadratic residues cannot be written as $S_2=A+B$ with $|A|,|B|\ge2$. The problem was resolved through three distinct breakthroughs.  Hanson and
Petridis~\cite{HP} proved the conjecture for almost all primes using a
Stepanov-type polynomial method, and Kalmynin~\cite{Kalmynin} subsequently
settled the remaining primes.  Most recently, Rudnev and
Tyrrell~\cite{RT26} classified all decompositions $S_d=A+B$ for proper
multiplicative subgroups of prime fields: apart from a singleton summand, the
only possibility is $|S_d|=4$ and $|A|=|B|=2$.  Related additive and multiplicative decomposition problems have been studied extensively; see, for example,~\cite{S12,S13,S14,S16,S20,KYY,KYY26}.

The present paper concerns the analogous question for restricted sumsets.  For $A\subset\F_p$, write
\[
        A\rsum A=\{a+a':a,a'\in A,\ a\ne a'\}.
\]
Restricted sumsets are classical objects in additive combinatorics.  The Erd\H{o}s--Heilbronn problem asks for the smallest possible size of $A\rsum A$; its sharp prime-field form was proved by Dias da Silva and Hamidoune~\cite{DH94}, with a later polynomial-method proof by Alon, Nathanson, and Ruzsa~\cite{ANR96}.  They also arise naturally in problems where every pair of distinct elements is required to have a prescribed arithmetic property, such as the Erd\H{o}s--Moser problem on sets whose pairwise sums are perfect squares~\cite{E63,Y25}, and in the study of cliques in Cayley sum graphs~\cite{Green05,GM16,Y25}.  The equation
\[
        S_d=A\rsum A
\]
therefore combines a classical restricted-addition problem with the additive decomposition of a highly multiplicatively structured set.

Several neighboring equations are already understood.  Lev and Sonn~\cite{LS17} initiated the study of representations by difference sets, and Kalmynin~\cite{Kalmynin} proved that $S_d\cup\{0\}=A-A$ is impossible except when $M=2$ or $M=6$. Apart from the trivial case, the equation $S_d=A+A$ can be ruled out more directly using Vinogradov's estimate for double character sums; see~\cite[Theorem~3.2]{Sh14} and~\cite[Proposition~4.1]{Y25}.  Neither result settles the restricted equation: the diagonal sums $2a$ are omitted from $A\rsum A$, so a representation $S_d=A\rsum A$ need not give an ordinary decomposition of $S_d$.  This omission is not merely technical.  One of our main structural results shows that, once $M\ge4$, such a representation forces all $\binom{|A|}{2}$ sums of distinct pairs to be different.  Thus a hypothetical set $A$ must exhibit simultaneously the additive rigidity of a Sidon set and the multiplicative symmetry of a subgroup.

This leads to the following restricted-sumset analogue of the generalized S\'ark\"ozy conjecture.

\begin{conjecture}\label{conj:main}
Let $d\ge2$.  If $p\equiv1\pmod d$ is sufficiently large, then there is no set $A\subset\F_p$ such that $A\rsum A=S_d$.
\end{conjecture}

For $d=2$, Shkredov~\cite{Sh14} proved the conjecture over prime fields for $p>13$.  His Fourier-analytic argument uses identities specific to the quadratic character to force any such set $A$ to be a perfect difference set, and then excludes this rigid possibility using the multiplier theorem and further character-sum identities.  These ingredients are particular to the quadratic setting and do not readily extend to general finite fields or to subgroups of higher index.

Yip~\cite[Theorem~1.1]{Y25} subsequently extended the quadratic case to every finite field $\F_q$ of odd order $q>13$.  His proof combined Stepanov's method with arithmetic input: over extension fields, the resulting numerical constraint led to a special case of the Nagell--Ljunggren equation~\cite{BM02}, while for higher-index subgroups, Kummer's theorem and a uniform-distribution result of Bergelson et al.~\cite{BKMST14} were used to control the exceptional primes.  More precisely, for each fixed $d\ge3$, Yip proved the conjecture for a set of primes of lower relative density at least $3/4$ among the primes $p\equiv1\pmod d$~\cite[Theorems~1.2 and~1.4]{Y25}.  He also proved that, when $M\ge3$, if $S_d=A\rsum A$ and either $|A|$ is odd or $0\in A$, then $A$ is Sidon and
\[
        2A\cap S_d=\varnothing,
        \qquad
        M=\binom{|A|}{2}.
\]
The case where $|A|$ is even and $0\notin A$ remained open, as did infinitely many possible characteristics in the higher-index problem.

We settle the problem completely over prime fields.

\begin{theorem}\label{thm:main}
Let $d\ge2$, and let $p\equiv1\pmod d$ be a prime with $p\ge7d+1$.
Then there is no set $A\subset\F_p$ such that $A\rsum A=S_d$.
\end{theorem}

Since $p-1=dM$, the hypothesis in \cref{thm:main} is exactly $M=|S_d|\ge7$.  The result is therefore uniform in both the characteristic and the index, has no asymptotic qualification, and gives the best possible threshold.  Indeed, together with the analysis of $M\le6$ in \cref{rem:small-exceptions}, it yields a complete classification: a proper multiplicative subgroup $S_d$ can be represented as $A\rsum A$ if and only if $M\in\{1,3,6\}$.  For $M=1$, the decompositions are the two-element sets $A=\{x,1-x\}$ with $x\ne\frac12$; for $M=3$, the unique decomposition is $A=-S_d$; and for $M=6$, the two decompositions are $A=\{0\}\sqcup H$ and $A=\{0\}\sqcup(-H)$, where $H$ is the subgroup of $S_d$ of order $3$.  No decomposition exists for $M=2,4,5$ or $M\ge7$.

We briefly describe the proof.  The first step is to convert a putative representation into a critical Sidon configuration.  Combining Stepanov's method, Lagrange interpolation, and Yip's diagonal estimate~\cite{Y25}, we show that, if $p\ge4d+1$ and $S_d=A\rsum A$, then $A$ is Sidon, and either $0\in A$ or $|A|$ is odd. We then treat the three remaining cases according to whether $0\in A$ and the parity of $|A|$.  Specially chosen auxiliary polynomials rule out the case $0\in A$ with $|A|$ odd.  In the other two cases, identities obtained from Hasse derivatives force $A$, or its nonzero part, to be a multiplicative coset; a final coefficient comparison leaves only the exceptional configurations of orders $3$ and $6$. This rigidity argument is uniform in $d$ and avoids both the
character-specific ingredients  and the Diophantine inputs used in earlier work.

\subsection*{Organization of the paper}
In \cref{sec:2}, we develop the Stepanov setup and prove the structural reduction to the Sidon case.  Sections~\ref{sec:3}, \ref{sec:4}, and \ref{sec:5} treat, respectively, the cases $0\in A$ with $|A|$ odd, $0\notin A$ with $|A|$ odd, and $0\in A$ with $|A|$ even.  Finally, in \cref{sec:6}, we prove \cref{thm:main} and classify all exceptional decompositions with $M\le6$.

\section{Stepanov setup and the critical case}\label{sec:2}

In this section, we prove the main reduction that will be used throughout the
rest of the paper.  Using Stepanov's method, together with a Lagrange
interpolation identity and Yip's diagonal estimate, we show that any
hypothetical decomposition $A\rsum A=S_d$ in the range $p\ge4d+1$  must be extremal: the set $A$ is Sidon, $M=\binom{|A|}{2}$, and $2A\cap S_d=\varnothing$, and moreover
either $0\in A$ or $|A|$ is odd.

For a polynomial $R(X) \in \F_p[X]$ and a point $x\in \F_p$, write $\E^{(r)}R(x)$ for the $r$-th Hasse derivative at $x$; equivalently,
\begin{equation}\label{eq:def_of_hasse}
  R(x+Y)=\sum_{r\ge0}\E^{(r)}R(x)Y^r
\end{equation}
in the auxiliary variable $Y$.
Thus $\E^{(r)}R(x)$ is the coefficient of $Y^r$ in the Taylor expansion at $x$.

We will use the standard fact that $R$ has a zero of multiplicity at least
$m$ at $x$ if and only if
\[
  \E^{(0)}R(x)=\E^{(1)}R(x)=\cdots=\E^{(m-1)}R(x)=0.
\]

Throughout the section, $R'$ denotes the formal derivative of $R$.

\begin{lemma}\label{lem:lagrange}
Let $T\subset\Fp$ have size $n$, and put
\[
  R_T(X)=\prod_{t\in T}(X-t), \qquad c_t=\frac1{R_T'(t)}.
\]
Then, for every polynomial $G\in\Fp[X]$ with $\deg G\le n-1$, we have
\[
  \sum_{t\in T} c_tG(t) = [X^{n-1}]G(X),
\]
where $[X^{n-1}]G(X)$ denotes the coefficient of $X^{n-1}$ in $G(X)$.
In particular,
\begin{equation}\label{eq:lagrange-identities}
  \sum_{t\in T}c_t t^j=0\quad \text{for every }\ 0\le j\le n-2, \qquad \sum_{t\in T}c_t t^{n-1}=1.
\end{equation}
Moreover,
\[
  \sum_{t\in T}\frac{c_t}{X-t}=\frac1{R_T(X)}.
\]
\end{lemma}

\begin{proof}
For each $t\in T$, set
\[
  L_t(X)=\frac{R_T(X)}{(X-t)R_T'(t)}.
\]
Then $L_t(t)=1$, while $L_t(u)=0$ for every $u\in T\setminus\{t\}$.
Thus, by Lagrange interpolation, every polynomial $G$ of degree at most
$n-1$ satisfies
\begin{equation}\label{eq:G(x)}
  G(X)=\sum_{t\in T}G(t)L_t(X).
\end{equation}
Since $R_T(X)/(X-t)$ is monic of degree $n-1$, the coefficient of
$X^{n-1}$ in $L_t(X)$ is $1/R_T'(t)=c_t$.  Taking the coefficient of
$X^{n-1}$ on both sides of \cref{eq:G(x)} gives the desired identity.

In particular, taking $G(X)=X^j$ gives $\sum_{t\in T}c_t t^j=[X^{n-1}]X^j$.
The right-hand side is $0$ for $0\le j\le n-2$, and is $1$ for
$j=n-1$. This gives \cref{eq:lagrange-identities}.

Finally, applying \cref{eq:G(x)} with $G(X)=1$ gives
\[
  1=\sum_{t\in T}\frac{R_T(X)}{(X-t)R_T'(t)}.
\]
Dividing both sides by $R_T(X)$ yields
\[ 
  \frac1{R_T(X)} = \sum_{t\in T}\frac1{(X-t)R_T'(t)} = \sum_{t\in T}\frac{c_t}{X-t}.
\]
This completes the proof.
\end{proof}

\begin{proposition}\label{prop:restricted-sum-sidon-prime}
Let $d\geq 2$, and let $p\equiv 1\pmod d$ be a prime with $p\ge 4d+1$.
If $A\subset\mathbb F_p$ satisfies $A\widehat{+}A=S_d$, 
then $0\in A$ or $|A|\ \text{is odd}$. 
Moreover, $A$ is Sidon in the sense that all restricted sums are distinct,
\[
\{2a:a\in A\}\cap S_d=\varnothing,
\qquad \text{and} \qquad M=\binom{|A|}{2}.
\]
\end{proposition}

\begin{proof}
Write $N=|A|$. We shall use the following diagonal-penalty estimate of Yip~\cite[Theorem 1.4]{Y25}, which says that if $A\widehat{+}A\subset S_d$ and either $N$ is odd or $0\in A$, then
\begin{equation}\label{eq:diagonal-penalty}
\binom{N}{2} + \#\{a\in A:2a\in S_d\} \leq M.
\end{equation}
We also note that since $A\widehat{+}A=S_d$, we have
\begin{equation}\label{ineq:n2m}
M=|S_d|=|A\widehat{+}A|\leq \binom{N}{2}.
\end{equation}

Now assume first that $N$ is odd or $0\in A$. 
Combining inequalities~\eqref{eq:diagonal-penalty} and \eqref{ineq:n2m}, we get $M=\binom{N}{2}$ 
and $\{a\in A:2a\in S_d\}=\varnothing$.
Thus we have $\{2a:a\in A\}\cap S_d=\varnothing$.
Also $|A\widehat{+}A|=\binom{N}{2}$, so all restricted sums $a+b$ with $a\ne b$ are distinct. Thus $A$ is Sidon.

It remains to rule out the remaining case $N$ even and $0\notin A$.
Write $N=2m+2$. Since $|A|\ge 4$, $m\ge1$. 
Note that $N\le M+1$. Indeed, since $A\widehat{+}A=S_d$, fixing any $a\in A$, the map $b\mapsto a+b$ sends $A\setminus\{a\}$ injectively into $S_d$. Thus $N-1=|A \setminus \{a\}| \le |S_d|=M$.
Also, since $p\ge4d+1$, we have $M\ge4$ and $p=dM+1\ge2M+1$. 
This implies that $m\le (M-1)/2$ and so $M+m<2M+1\le p$.

Let $A=\{a_1,\ldots,a_N\}$, and let
\[
P(X)=\prod_{i=1}^N(X-a_i) \qquad \text{and} \qquad  c_i=\frac{1}{P'(a_i)}.
\]
Then by \cref{lem:lagrange} applied to $T=A$,
\begin{equation}\label{eq:propci}
\sum_{i=1}^N c_i a_i^j=0
\quad \text{for every } \ 0\leq j\leq N-2, \qquad \sum_{i=1}^N c_i a_i^{N-1}=1.
\end{equation}
Consider the auxiliary polynomial
\[
f(X)= -(-1)^{m+1} + \sum_{i=1}^N c_i(X+a_i)^{M+m}(X-a_i)^{m+1}.
\]
For $0\le \ell\le 2m=N-2$, the coefficient of $X^{M+2m+1-\ell}$ in $f$ is
\[
  \left(\sum_{u+v=\ell}\binom{M+m}u\binom{m+1}v(-1)^v\right) \cdot \sum_{i=1}^N c_i a_i^\ell .
\]
By \cref{eq:propci}, the coefficient of $X^{M+2m+1-\ell}$ in $f$ vanishes for every $0\le\ell\le N-2=2m$. Thus $\deg f\le M+2m+1-(2m+1)=M$.

Next we fix $a_j\in A$ for some $1 \le j \le N$ and prove that $\E^{(r)}f(a_j)=0$ for $0\le r\le m$.  
For $i\ne j$, since $a_j+a_i\in A\widehat{+}A=S_d$, we have $(a_j+a_i)^M=1$. For $i\ne j$, by \cref{eq:def_of_hasse}, the $r$-th Hasse derivative at $X=a_j$ of the $i$-th summand $(X+a_i)^{M+m}(X-a_i)^{m+1}$ is
\begin{align*}
&\sum_{u=0}^r   \binom{M+m}u\binom{m+1}{r-u} (a_j+a_i)^{M+m-u}(a_j-a_i)^{m+1-r+u} \\
& \qquad = \sum_{u=0}^r \binom{M+m}u\binom{m+1}{r-u} (a_j+a_i)^{m-u}(a_j-a_i)^{m+1-r+u}.
\end{align*}
As a function of $a_i$, this is a polynomial of degree at most
$2m+1-r$. We denote it by $R_{r,j}(a_i)$. If $1\le r\le m$, then $2m+1-r\le2m=N-2$, and \cref{lem:lagrange} gives $\sum_{i=1}^N c_iR_{r,j}(a_i)=0$.
Also $R_{r,j}(a_j)=0$, since $m+1-r+u>0$ for $0\le u\le r\le m$, and the diagonal summand $c_j(X+a_j)^{M+m}(X-a_j)^{m+1}$ is divisible by $(X-a_j)^{m+1}$.
Thus $\E^{(r)}f(a_j)=0$ for $1\le r\le m$.
For $r=0$, the polynomial $ R_{0,j}(T)=(a_j+T)^m(a_j-T)^{m+1}$ has degree $2m+1=N-1$ and leading coefficient $(-1)^{m+1}$. Thus by \cref{lem:lagrange}, $\sum_{i=1}^Nc_iR_{0,j}(a_i)=(-1)^{m+1}$.
Moreover $R_{0,j}(a_j)=0$, and the diagonal summand $c_j(a_j+a_j)^{M+m}(a_j-a_j)^{m+1}$ in $f(a_j)$ is also zero. 
Therefore 
\[f(a_j)= -(-1)^{m+1} + \sum_{i=1}^N c_iR_{0,j}(a_i) +0 = 0.\]

We have proved that $f$ has a zero of multiplicity at least $m+1$ at each of the
$N$ distinct points $a_1,\ldots,a_N$.
This implies that the product $\prod_{j=1}^N (X-a_j)^{m+1}$ divides $f(X)$.  If $f\not\equiv0$, then $N(m+1)\le \deg f$.
Since $\deg f\le M$, we obtain $N(m+1)\le M$.
But, by \cref{ineq:n2m}, we have
\[
N(m+1)=2(m+1)^2 > (m+1)(2m+1)=\binom{N}{2}\geq M,
\]
a contradiction. Thus $f\equiv0$.

We may now compute the next two Hasse coefficients $\E^{(m+1)}f(a_j)$ and $\E^{(m+2)}f(a_j)$ to derive a contradiction.
Define
\[
S_1(j)=\sum_{i=1}^N\frac{c_i}{a_j+a_i} \qquad \text{and} \qquad S_2(j)=\sum_{i=1}^N\frac{c_i}{(a_j+a_i)^2}.
\]
The denominators are nonzero. Indeed, if $i\ne j$, then $a_i+a_j\in S_d\subset\mathbb F_p^*$, while if $i=j$, then $2a_j\ne0$ since $0\notin A$.
Also, put $C_1=\binom{M+m}{m+1}$ and $C_2=\binom{M+m}{m+2}$.

We first compute $\E^{(m+1)}f(a_j)$.  For the coefficient of
$Y^{m+1}$ in the $i$-th summand at $X=a_j$, write $ s=a_j+a_i$ and $t=a_j-a_i$.
Then this coefficient is
\[
  c_i\sum_{u=0}^{m+1}\binom{M+m}{u}\binom{m+1}{m+1-u}s^{M+m-u}t^u.
\]

Assume first that $i\ne j$.  Then $s=a_j+a_i\in S_d$, and hence
$s^M=1$.  Thus, for $0\le u\le m$, the corresponding term becomes
\[
  c_i\binom{M+m}{u}\binom{m+1}{m+1-u} s^{m-u}t^u.
\]
For fixed $u$, the factor $s^{m-u}t^u=(a_j+a_i)^{m-u}(a_j-a_i)^u$ is the value at $a_i$ of a polynomial in $X$ of degree at most $m$.
Since $m\le N-2$, \cref{lem:lagrange} gives $\sum_{i=1}^N c_i(a_j+a_i)^{m-u}(a_j-a_i)^u=0$.
If $1\le u\le m$, then the diagonal value $i=j$ is zero, since
$(a_j-a_j)^u=0$.  Thus the contributions with $1\le u\le m$
vanish after summing over $i\ne j$.

The case $u=0$ has to be kept separately.  By the same Lagrange
identity, $\sum_{i\ne j} c_i(a_j+a_i)^m = -c_j(2a_j)^m$.
On the other hand, the diagonal summand $i=j$ contributes $c_j(2a_j)^{M+m}$ to the coefficient of $Y^{m+1}$, since $c_j(2a_j+Y)^{M+m}Y^{m+1}$ has $Y^{m+1}$-coefficient $c_j(2a_j)^{M+m}$.  Therefore the combined contribution of the $u=0$ term and the diagonal summand is
$c_j(2a_j)^m\big((2a_j)^M-1\big)$.

It remains to consider the term $u=m+1$ with $i\ne j$.  This contributes
\[
  C_1\sum_{i\ne j}c_i\frac{(a_j-a_i)^{m+1}}{a_j+a_i}.
\]
The summand vanishes at $i=j$, so we may include the diagonal index.  Using $  a_j-a_i=2a_j-(a_j+a_i)$,
we get
\[
  \frac{(a_j-a_i)^{m+1}}{a_j+a_i} = \frac{(2a_j)^{m+1}}{a_j+a_i}
  + \text{a polynomial in } a_i \text{ of degree at most } m.
\]
The polynomial part vanishes after summing against the weights $c_i$, by
\cref{lem:lagrange}.  Thus
\[
  C_1\sum_{i\ne j}c_i\frac{(a_j-a_i)^{m+1}}{a_j+a_i} = C_1(2a_j)^{m+1}S_1(j).
\]
Consequently, $\E^{(m+1)}f(a_j) = c_j(2a_j)^m\big((2a_j)^M-1\big) + C_1(2a_j)^{m+1}S_1(j)$.
Since this Hasse coefficient must vanish due to $f \equiv 0$, we obtain
\begin{equation}\label{eq:first-even-derivative}
  c_j(2a_j)^m\big((2a_j)^M-1\big) + C_1(2a_j)^{m+1}S_1(j)=0.
\end{equation}

We next compute $\E^{(m+2)}f(a_j)$.  Again write $s=a_j+a_i$ and $t=a_j-a_i$.
The coefficient of $Y^{m+2}$ in the $i$-th summand is
\[
  c_i\sum_{u=1}^{m+2} \binom{M+m}{u}\binom{m+1}{m+2-u} s^{M+m-u}t^{u-1}.
\]
Here the sum starts at $u=1$, since $\binom{m+1}{m+2}=0$.

Assume first that $i\ne j$.  Then $s\in S_d$, and so $s^M=1$.
For $1\le u\le m$, the corresponding term becomes
\[
  c_i\binom{M+m}{u}\binom{m+1}{m+2-u}s^{m-u}t^{u-1}.
\]
For fixed $u$, the factor $s^{m-u}t^{u-1} = (a_j+a_i)^{m-u}(a_j-a_i)^{u-1}$ is the value at $a_i$ of a polynomial in $X$ of degree at most $(m-u)+(u-1)=m-1$.
Thus its weighted sum over all $i$ vanishes by \cref{lem:lagrange}.
If $2\le u\le m$, then the diagonal value $i=j$ is also zero, since
$(a_j-a_j)^{u-1}=0$.  Therefore the contributions with
$2\le u\le m$ vanish after summing over $i\ne j$.

The case $u=1$ must be kept separately.  By the same Lagrange identity,
\[
  \sum_{i\ne j}c_i(a_j+a_i)^{m-1} = -c_j(2a_j)^{m-1}.
\]
Thus the $u=1$ contribution from the terms $i\ne j$ is $-(M+m)c_j(2a_j)^{m-1}$.
On the other hand, the diagonal summand $i=j$ contributes $(M+m)c_j(2a_j)^{M+m-1}$ to the coefficient of $Y^{m+2}$, since $c_j(2a_j+Y)^{M+m}Y^{m+1}$ has $Y^{m+2}$-coefficient $(M+m)c_j(2a_j)^{M+m-1}$.
Therefore the combined contribution of the $u=1$ term and the diagonal
summand is
\[
  (M+m)c_j(2a_j)^{m-1}\big((2a_j)^M-1\big).
\]

It remains to consider the terms $u=m+1$ and $u=m+2$.  The term
$u=m+1$ contributes
\[
  (m+1)C_1 \sum_{i\ne j}c_i\frac{(a_j-a_i)^m}{a_j+a_i}.
\]
The summand vanishes at $i=j$, since $m\ge1$, so we may include the
diagonal index.  Using $a_j-a_i=2a_j-(a_j+a_i)$,
we have
\[
  \frac{(a_j-a_i)^m}{a_j+a_i} = \frac{(2a_j)^m}{a_j+a_i} + \text{a polynomial in } a_i \text{ of degree at most } m-1.
\]
The polynomial part vanishes after summing against the weights $c_i$, by
\cref{lem:lagrange}.  Thus the $u=m+1$ contribution is $(m+1)C_1(2a_j)^m S_1(j)$.

Finally, the term $u=m+2$ contributes
\[
  C_2 \sum_{i\ne j}c_i\frac{(a_j-a_i)^{m+1}}{(a_j+a_i)^2}.
\]
Again the summand vanishes at $i=j$, so we may include the diagonal index.
Using $a_j-a_i=2a_j-(a_j+a_i)$, we get
\[
  \frac{(a_j-a_i)^{m+1}}{(a_j+a_i)^2} = \frac{(2a_j)^{m+1}}{(a_j+a_i)^2} - (m+1)\frac{(2a_j)^m}{a_j+a_i} + \text{a polynomial in } a_i \text{ of degree at most } m-1.
\]
The polynomial part again vanishes by \cref{lem:lagrange}.  Therefore the
$u=m+2$ contribution is
\[
  C_2(2a_j)^{m+1}S_2(j) - (m+1)C_2(2a_j)^m S_1(j).
\]

Combining the three remaining contributions, we obtain
\[
\begin{aligned}
  \E^{(m+2)}f(a_j) &= (M+m)c_j(2a_j)^{m-1}\big((2a_j)^M-1\big)  \\
  &\quad +(m+1)(C_1-C_2)(2a_j)^mS_1(j)
  +C_2(2a_j)^{m+1}S_2(j).
\end{aligned}
\]
Since this Hasse coefficient must vanish, we have
\begin{equation}\label{eq:second-even-derivative}
\begin{aligned}
  &(M+m)c_j(2a_j)^{m-1}\big((2a_j)^M-1\big)  \\
  &\qquad +(m+1)(C_1-C_2)(2a_j)^mS_1(j)+C_2(2a_j)^{m+1}S_2(j)=0.
\end{aligned}
\end{equation}

Eliminating $c_j(2a_j)^m((2a_j)^M-1)$ from equations~\eqref{eq:first-even-derivative} and~\eqref{eq:second-even-derivative} gives
\[
  2C_2a_jS_2(j)=\bigl((M+m)C_1-(m+1)(C_1-C_2)\bigr)S_1(j).
\]
Since $C_1/C_2=(m+2)/(M-1)$, the coefficient on the right divided by $C_2$ is
\[
  (M+m)\frac{m+2}{M-1}-(m+1)\left(\frac{m+2}{M-1}-1\right)=2m+3.
\]
Thus $2a_jS_2(j)=(2m+3)S_1(j)$.
Since $N=2m+2$, this is
\begin{equation}\label{eq:S1S2-relation}
2a_jS_2(j)=(N+1)S_1(j).
\end{equation}

Now express $S_1,S_2$ in terms of the polynomial $P$. The partial fraction identity gives
\[
\sum_{i=1}^N\frac{c_i}{X+a_i} = -\frac{1}{P(-X)}.
\]
Let $Q(X)=P(-X)$.
Then
\[
S_1(X):=\sum_{i=1}^N\frac{c_i}{X+a_i} = -\frac1{Q(X)} \quad \text{and} \quad S_2(X):=\sum_{i=1}^N\frac{c_i}{(X+a_i)^2} = -S_1'(X) = -\frac{Q'(X)}{Q(X)^2}.
\]
Thus we have $S_2(X)/S_1(X)=Q'(X)/Q(X)$.
Substituting $X=a_j$ into equation~\eqref{eq:S1S2-relation} for each $1\leq j\leq N$, we get $2a_jQ'(a_j)=(N+1)Q(a_j)$.
Therefore the polynomial \[H(X):=2XQ'(X)-(N+1)Q(X)\] vanishes at every element of $A$, i.e. at every root of $P$. Since $N$ is even, $Q(X)=P(-X)$ is monic of degree $N$. The leading coefficient of $H$ is $2N-(N+1)=N-1$.
Since $N\le M+1<p$, this is nonzero in $\mathbb F_p$. Thus $H(X)=(N-1)P(X)$.

Write $P(X)=P_{\rm ev}(X)+P_{\rm odd}(X)$, where $P_{\rm ev}$ and $P_{\rm odd}$ are the even and odd parts of $P$. Since $N$ is even, $Q(X)=P(-X)=P_{\rm ev}(X)-P_{\rm odd}(X)$.
Comparing even and odd parts in $2XQ'(X)-(N+1)Q(X)=(N-1)P(X)$
gives
\[
XP_{\rm ev}'(X)=NP_{\rm ev}(X) \qquad \text{and} \qquad XP_{\rm odd}'(X)=P_{\rm odd}(X).
\]
Since $\deg P=N<p$, these differential equations imply $P_{\rm ev}(X)=X^N$ and $P_{\rm odd}(X)=\beta X$ for some $\beta\in\mathbb F_p$. Thus $P(X)=X^N+\beta X=X(X^{N-1}+\beta)$,
so $0$ is a root of $P$. This contradicts $0\notin A$.

Therefore, the case $N$ even and $0\notin A$ is impossible. Thus $0\in A$ or $N$ is odd. In either case, the first part of the proof applies, and thus the proposition follows.
\end{proof}

\section[The zero-odd case]{\texorpdfstring{The zero-odd case: $0\in A$ and $|A|$ is odd}{The zero-odd case}}\label{sec:3}

In this section, we handle the case in which the set contains zero and has odd
cardinality.  


\begin{proposition}\label{prop:zero-odd}
Let $d\ge2$, and let $p\equiv1\pmod d$ be a prime with $p\ge4d+1$.  There is no set $A\subset\Fp$ such that $0\in A$, $|A|$ is odd, and $A\rsum A=S_d$.
\end{proposition}

\begin{proof}
Assume, for contradiction, that such a set $A$ exists.  Write $A=\{0\}\sqcup B$, $|A|=2m+1$, and $|B|=2m$.
By \cref{prop:restricted-sum-sidon-prime}, the set $A$ is Sidon,
$2A\cap S_d=\varnothing$, and $M=\binom{2m+1}{2}=m(2m+1)$.
In particular $m\ge1$.  Since $0\in A$, we have $B\subset S_d$.  The
Sidon property gives the disjoint partition $S_d=B\sqcup(B\rsum B)$, and also $2B\cap S_d=\varnothing$.
Put
\[
  Q(X)=\prod_{b\in B}(X-b) \qquad \text{and} \qquad c_b=\frac1{Q'(b)}.
\]
Define
\begin{equation}\label{eq:zero-odd-Phi}
  \Phi(X) = -(-1)^m+\sum_{b\in B}c_b(X+b)^{M+m-1}(X-b)^m .
\end{equation}

We first record two consequences of the same argument used in the proof of \cref{prop:restricted-sum-sidon-prime}.  First, $\deg\Phi\le M$.
Indeed, for $0\le \ell\le2m-2=|B|-2$, the coefficient of $X^{M+2m-1-\ell}$ in the sum defining $\Phi$ is a linear combination of $\sum_{b\in B}c_b b^j$ with $ 0\le j\le \ell$, and all these sums vanish by \cref{lem:lagrange}.  Thus the coefficients of $X^{M+2m-1},\ldots,X^{M+1}$ in the sum defining $\Phi$ are all zero. Thus $\deg \Phi\le M+2m-1-(2m-1)=M.$.

Second, we claim that $X^{2m}Q(X)^m\mid \Phi(X)$. 
For $X^{2m}\mid \Phi(X)$, use $b^M=1$ for $b\in B$ and write
\[
  (X+b)^{M+m-1}(X-b)^m = (-1)^m b^{2m-1} \left(1+\frac Xb\right)^{M+m-1}   \left(1-\frac Xb\right)^m .
\]
If $(1+T)^{M+m-1}(1-T)^m=\sum_{r\ge0}h_rT^r$, 
then
\[
  \Phi(X)+(-1)^m = (-1)^m\sum_{r\ge0}h_rX^r \sum_{b\in B}\frac{b^{2m-1-r}}{Q'(b)}.
\]
For $r=0$, we have $h_0=1$, and by \cref{lem:lagrange},
$\sum_{b\in B}\frac{b^{2m-1}}{Q'(b)}=1$.  Thus the $r=0$ is exactly cancelled by the constant term $-(-1)^m$. Also, the terms $1\le r\le2m-1$ vanish by \cref{lem:lagrange}.  Thus $X^{2m}\mid\Phi(X)$.

It remains to show that $Q(X)^m\mid\Phi(X)$.  Fix $b_0\in B$.  
We prove that $\E^{(r)}\Phi(b_0)=0$ for $0\le r\le m-1$.  For $\beta\ne b_0$, we have $b_0+\beta\in B\rsum B\subset S_d$, so $(b_0+\beta)^M=1$.  Thus, for $0\le r\le m-1$,
\[
\E^{(r)}\bigl((X+\beta)^{M+m-1}(X-\beta)^m\bigr)\big|_{X=b_0} =\sum_{i=0}^r
  \binom{M+m-1}{i}\binom{m}{r-i}
  (b_0+\beta)^{m-1-i}(b_0-\beta)^{m-r+i}.
\]
The last expression is a polynomial $R_r(\beta)$ of degree at most $(m-1-i)+(m-r+i)=2m-r-1$.
If $1\le r\le m-1$, then $2m-r-1\le2m-2=|B|-2$.
Thus \cref{lem:lagrange}, applied to $T=B$ and $G=R_r$, gives $\sum_{\beta\in B}R_r(\beta)/Q'(\beta)=0$.
Moreover $R_r(b_0)=0$, since the exponent $m-r+i$ is positive for every $0\le i\le r<m$.
Thus $\sum_{\beta\ne b_0}R_r(\beta)/Q'(\beta)=0$.
This is exactly the total contribution of the off-diagonal terms to  $\E^{(r)}\Phi(b_0)$.  The diagonal summand is divisible by $(X-b_0)^m$, so its $r$-th Hasse derivative at $b_0$ is also zero.
Thus $\E^{(r)}\Phi(b_0)=0$ for $1\le r\le m-1$.

For $r=0$, the polynomial is $R_0(\beta)=(b_0+\beta)^{m-1}(b_0-\beta)^m$.
It has degree $2m-1=|B|-1$ and leading coefficient $(-1)^m$.  Thus, by \cref{lem:lagrange},  $\sum_{\beta\in B}R_0(\beta)/Q'(\beta)=(-1)^m$.
The diagonal term is still zero, and the constant term $-(-1)^m$ in equation~\eqref{eq:zero-odd-Phi} cancels this value.  Thus $\E^{(0)}\Phi(b_0)=\Phi(b_0)=0$.  We have proved $Q(X)^m\mid \Phi(X)$.

\smallskip
Since $X^{2m}$ and $Q(X)$ are coprime, $X^{2m}Q(X)^m\mid\Phi(X)$.
The degree of the divisor is $2m+m|B|=2m+2m^2$, whereas $\deg\Phi\le M=m(2m+1)=2m^2+m$.
For $m\ge1$, the divisor $X^{2m}Q(X)^m$ has degree larger than $M$, so $\Phi\equiv0$.

\smallskip
Now fix $b\in B$, and set $C=\binom{M+m-1}{m}$.
Since $\Phi\equiv0$, we have
$\E^{(m)}\Phi(b)=0$.
We compute this coefficient explicitly, that is, the coefficient of $Y^m$ in $\Phi(b+Y)$.
Recall that $\Phi(X)$ in \cref{eq:zero-odd-Phi}. 
For $\beta\in B$, put $s=b+\beta$ and $t=b-\beta$.
Then, after substituting $X=b+Y$, we have $X+\beta=s+Y$ and $X-\beta=t+Y$.
Thus the coefficient of $Y^m$ in the $\beta$-summand is
\[
  \frac1{Q'(\beta)} \sum_{i=0}^m \binom{M+m-1}{i} \binom{m}{m-i} s^{M+m-1-i}t^i.
\]
We split this sum over $i$ into the range $0\le i\le m-1$ and the last term $i=m$.

First consider the part with $0\le i\le m-1$.  For $\beta\ne b$, we have $s=b+\beta\in B\widehat{+}B\subset S_d$,
and so $s^M=1$.  Thus this part becomes $G(\beta)/Q'(\beta)$, 
where
\[
  G(T)= \sum_{i=0}^{m-1} \binom{M+m-1}{i} \binom{m}{m-i} (b+T)^{m-1-i}(b-T)^i.
\]
The polynomial $G(T)$ has degree at most $m-1$.  Since
$m-1\le 2m-2=|B|-2$, \cref{lem:lagrange}, applied with $T=B$, gives $\sum_{\beta\in B}G(\beta)/Q'(\beta)=0$.

At the diagonal point $\beta=b$, $ G(b)=(2b)^{m-1}$. Thus the
off-diagonal contribution from the range $0\le i\le m-1$ is
\[
  \sum_{\substack{\beta\in B\\ \beta\ne b}} \frac{G(\beta)}{Q'(\beta)} = -\frac{(2b)^{m-1}}{Q'(b)}.
\]
On the other hand, after substituting $X=b+Y$, the diagonal summand is
\[
  \frac{(X+b)^{M+m-1}(X-b)^m}{Q'(b)}=\frac{(2b+Y)^{M+m-1}Y^m} {Q'(b)}.
\]
Thus the coefficient of $Y^m$ is $(2b)^{M+m-1}/Q'(b) = (2b)^{m-1}(2b)^M/Q'(b)$.
Combining this diagonal contribution with the off-diagonal contribution, the total contribution of the range $0\le i\le m-1$ is
\[
  -\frac{(2b)^{m-1}}{Q'(b)} + \frac{(2b)^{m-1}(2b)^M}{Q'(b)} = \frac{(2b)^{m-1}((2b)^M-1)}{Q'(b)}.
\]

It remains to consider the last term $i=m$.  For $\beta\ne b$, using $s^M=1$ and $t=b-\beta=2b-s$, this term is
\[
  \binom{M+m-1}{m}s^{M-1}t^m= C\frac{(2b-s)^m}{s}=C\frac{(2b)^m}{s} +C\sum_{j=1}^m \binom mj (2b)^{m-j}(-s)^{j-1}.
\]
The second summand on the right is a polynomial in $s$, so in $\beta$, of degree at most $m-1$.  By \cref{lem:lagrange}, its contribution to the sum over $\beta\in B$ is zero.  
The diagonal index may be included in this sum, since the reduced expression vanishes at $\beta=b$.
Thus the contribution of the last term $i=m$ is
\[
  C(2b)^m\sum_{\beta\in B}\frac1{Q'(\beta)(b+\beta)} = C(2b)^m\sum_{\beta\in B}\frac{c_\beta}{b+\beta}.
\]
Therefore, combining this with the contribution from the range
$0\le i\le m-1$, we obtain
\[
  \E^{(m)}\Phi(b) = \frac{(2b)^{m-1}((2b)^M-1)}{Q'(b)} + C(2b)^m\sum_{\beta\in B}\frac{c_\beta}{b+\beta}.
\]
Since $\Phi\equiv0$ and   $\sum_{\beta\in B}\frac{c_\beta}{b+\beta}  = -1/Q(-b)$, this gives
\begin{equation}\label{eq:zero-odd-local}
  \frac{(2b)^{m-1}((2b)^M-1)}{Q'(b)} - \frac{C(2b)^m}{Q(-b)}   = 0.
\end{equation}
Since $B\subset S_d$, we have $b^M=1$.  Also $2B\cap S_d=\varnothing$, so $2\notin S_d$, and so $2^M\ne1$.  Thus, $(2b)^M=2^Mb^M=2^M\ne1$.
Combining this with \cref{eq:zero-odd-local}, we get
\[
  Q(-b) = \frac{2bCQ'(b)}{(2b)^M-1} = \frac{2C}{2^M-1}\,bQ'(b).
\]
Set $\alpha=(2C)/(2^M-1)$.
Then, for every $b\in B$, $Q(-b)=\alpha bQ'(b)$.

Define $R(X):=Q(-X)-\alpha XQ'(X)$.
Then $R(b)=0$ for every $b\in B$.
Since $Q(X)=\prod_{b\in B}(X-b)$, it follows that $Q\mid R$.
Moreover, $\deg R\le \deg Q=2m$.
Thus there exists $\lambda\in\Fp$ such that $R=\lambda Q$.

We compare the leading coefficients of $R$ and $Q$.  Since $Q$ is monic of even degree $2m$, the polynomial $Q(-X)$ also has leading coefficient $1$.  Also the leading term of $XQ'(X)$ is $2mX^{2m}$.
Therefore the leading coefficient of $R(X)=Q(-X)-\alpha XQ'(X)$ is $1-2m\alpha$.
Since $R=\lambda Q$ and $Q$ is monic, we get $\lambda=1-2m\alpha$.

On the other hand, evaluating $R=\lambda Q$ at $X=0$ gives $R(0)=\lambda Q(0)$.
But by the definition of $R$, $R(0)=Q(0)$.
Since $0\notin B$, $Q(0)=\prod_{b\in B}(-b)\ne0$.
Thus $\lambda=1$.  Combining this with $\lambda=1-2m\alpha$ gives $2m\alpha=0$.
However, since $2m<p$, we have $2m\ne0$ in $\Fp$.  Also $C\ne0$ in $\Fp$, since $M+m-1<p$, and $2^M-1\ne0$.  Thus $\alpha=(2C)/(2^M-1)\ne0$.
Thus $2m\alpha\ne0$, a contradiction.  This completes the proof.
\end{proof}

\section[The nonzero-odd case]{\texorpdfstring{The odd nonzero case: $0\notin A$ and $|A|$ is odd}{The odd nonzero case}}\label{sec:4}


In this section, we consider the case when the set contains no zero and has odd cardinality. 

Let $|A|=N=2m+1$.  By Proposition~\ref{prop:restricted-sum-sidon-prime}, in this case
\[
  M=\binom N2=mN, \qquad A\text{ is Sidon}, \qquad 2A\cap S_d=\varnothing.
\]
The proof has two steps.  First we show that $A$ must be a multiplicative coset.  Then we rule out that coset for $m\ge2$.

\begin{proposition}\label{prop:odd-nonzero-coset}
Let $d\ge2$, and let $p\equiv1\pmod d$ be a prime with $p\ge4d+1$.
Assume $A\rsum A=S_d$, $0\notin A$, and $|A|=N=2m+1$.  Let
\[
  P(X)=\prod_{a\in A}(X-a).
\]
Then $P(X)=X^N+\beta$ for some $\beta\in\Fp^*$.  Thus $A$ is a multiplicative coset.
\end{proposition}

\begin{proof}
Put $c_a=1/P'(a)$ and define
\begin{equation}\label{eq:G-shifted}
  G(X)=\sum_{a\in A}c_a(X+a)^{M+m-1}(X-a)^m.
\end{equation}

We first bound the degree of $G$. 
The coefficient of $X^{M+2m-1-\ell}$ in $G$ is a linear combination of
$\sum_{a\in A}c_a a^j$  with $0\le j\le\ell$.
For $0\le\ell\le2m-1=N-2$, all these sums vanish by equation~\eqref{eq:lagrange-identities}.  Thus $\deg G\le M+2m-1-2m=M-1$.

Note that by the same argument as in the proofs of the preceding propositions, one obtains $P(X)^m\mid G(X)$.
We omit the repeated details. 
Since $\deg P^m=mN=M$ and $\deg G\le M-1$, we get $G\equiv0$.

Fix $b\in A$ and put  
\[
  S_1(b)=\sum_{a\in A}\frac{c_a}{b+a} \qquad \text{and} \qquad S_2(b)=\sum_{a\in A}\frac{c_a}{(b+a)^2}.
\]
If $a\ne b$, then $a+b\in S_d\subset\Fp^*$, while if $a=b$, then $2b\ne0$.
Thus all denominators are nonzero.
Set $C=\binom{M+m-1}{m}$ and $D=\binom{M+m-1}{m+1}$.
Since $G\equiv0$, both $\E^{(m)}G(b)$ and $\E^{(m+1)}G(b)$ vanish.

For $\E^{(m)}G(b)$, the terms $0\le i\le m-1$ contribute $(2b)^{m-1}((2b)^M-1)/P'(b)$, as in the proof of \cref{prop:zero-odd}.  The term $i=m$ is
\[
  C\frac{(2b-s)^m}{s} =C\left(\frac{(2b)^m}{s}+\sum_{j=1}^m\binom mj(2b)^{m-j}(-s)^{j-1}\right),
\]
so its polynomial part cancels by \cref{lem:lagrange}, and its rational part contributes $C(2b)^mS_1(b)$.  Since $G\equiv0$, we have $\E^{(m)}G(b)=0$, and therefore
\begin{equation}\label{eq:G-local-m}
  0=\frac{(2b)^{m-1}((2b)^M-1)}{P'(b)}+C(2b)^mS_1(b).
\end{equation}

For $\E^{(m+1)}G(b)$, the polynomial terms with $1\le i\le m-1$ together with the diagonal term give
\[
  \frac{(M+m-1)(2b)^{m-2}((2b)^M-1)}{P'(b)}.
\]
The remaining non-polynomial terms come from $i=m$ and $i=m+1$, which are
\[
  mC\frac{(2b-s)^{m-1}}s=mC\frac{(2b)^{m-1}}s+g(s) \quad \text{and} \quad 
  D\frac{(2b-s)^m}{s^2}=D\frac{(2b)^m}{s^2}-D\frac{m(2b)^{m-1}}s+h(s),
\]
for some polynomial $g,h$ in $s$, and cancelling the polynomial parts by \cref{lem:lagrange}, their total contribution is $m(C-D)(2b)^{m-1}S_1(b)+D(2b)^mS_2(b)$.
Since $G\equiv0$, we have $\E^{(m+1)}G(b)=0$, and thus
\begin{equation}\label{eq:G-local-m+1}
  0=\frac{(M+m-1)(2b)^{m-2}((2b)^M-1)}{P'(b)}
  +m(C-D)(2b)^{m-1}S_1(b)+D(2b)^mS_2(b).
\end{equation}

Eliminating $(2b)^M-1$ from equations~\eqref{eq:G-local-m} and~\eqref{eq:G-local-m+1} gives $2bDS_2(b)=\bigl(C(M-1)+mD\bigr)S_1(b)$.
Since $D=C(M-1)/(m+1)$,
this is reduced to $2bS_2(b)=NS_1(b)$.
By \cref{lem:lagrange}, $\sum_{a\in A}c_a/(X+a)=-1/P(-X)$.
Thus
\[
  S_1(b)=-\frac1{P(-b)} \qquad \text{and} \qquad  S_2(b)=\frac{P'(-b)}{P(-b)^2}.
\]
Substituting these into the equation $2bS_2(b)=NS_1(b)$ gives $2bP'(-b)+NP(-b)=0$ for all $b\in A$.
Thus the polynomial $H(X)=2XP'(-X)+NP(-X)$ vanishes at every root of $P$.  It has leading coefficient $N$, since $N$ is odd and $P$ is monic of degree $N$.  Thus $H(X)=NP(X)$.
Equivalently,
\begin{equation}\label{eq:odd-differential-equation}
  2XP'(-X)+NP(-X)=NP(X).
\end{equation}
Write $P=P_{\rm ev}+P_{\rm odd}$ for the even and odd parts.  Since $N$ is odd,
\[
  P(-X)=P_{\rm ev}(X)-P_{\rm odd}(X) \qquad \text{and} \qquad P'(-X)=-P'_{\rm ev}(X)+P'_{\rm odd}(X).
\]
Substitution into equation~\eqref{eq:odd-differential-equation} gives $-XP'_{\rm ev}(X)+XP'_{\rm odd}(X)=NP_{\rm odd}(X)$.
The even and odd parts therefore satisfy $XP'_{\rm ev}(X)=0$ and $XP'_{\rm odd}(X)=NP_{\rm odd}(X)$.
This, together with $\deg P=N<p$, implies that $P_{\rm ev}$ is constant and $P_{\rm odd}$ is a scalar multiple of $X^N$.  As $P$ is monic of odd degree $N$,
$P(X)=X^N+\beta$.
Since $0\notin A$, $\beta=P(0)\ne0$.  Therefore the roots of $P$ form a multiplicative coset.
\end{proof}

\begin{proposition}\label{prop:odd-coset-contradiction}
Let $d\ge2$, and let $p\equiv1\pmod d$ be a prime with $p\ge4d+1$.  There is no decomposition $A\rsum A=S_d$ with $0\notin A$, $|A|=2m+1$, and $m\ge2$.
\end{proposition}

\begin{proof}
Let $N=|A|$. By \cref{prop:odd-nonzero-coset}, write $A=\alpha H$, where $H\le\Fps$ has order $N=2m+1$.  Then
\[
  P(X)=\prod_{a\in A}(X-a)=X^N-\alpha^N,
\]
and $M=mN$.  For every $a\in A$, $(2a)^M$ is independent of $a$ since $a^M=(\alpha h)^M=\alpha^M(h^N)^m=\alpha^M$.  Moreover $(2a)^M\ne1$ since $2A\cap S_d=\varnothing$.

Define the odd auxiliary polynomial
\begin{equation}\label{eq:corrected-odd-F}
  \widetilde F(X)=-(-1)^m + \sum_{a\in A}\frac{(X+a)^{M+m}(X-a)^m}{P'(a)}.
\end{equation}
We first show that $P(X)^m\mid\widetilde F(X)$.  Consider the coefficient
of $X^{M+2m-\ell}$ in the sum in \eqref{eq:corrected-odd-F}.  It is a scalar multiple of $\sum_{a\in A}\frac{a^\ell}{P'(a)}$.
For $0\le\ell\le2m-1=N-2$, this sum is zero by \cref{lem:lagrange}
applied to $T=A$.  Thus the top $2m$ coefficients vanish.
Thus $\deg\widetilde F\le M+2m-2m=M$.
Fix $b\in A$ and $0\le r\le m-1$.  For $a\ne b$, after writing $s=b+a$ and using $s^M=1$, the $r$-th Hasse derivative of the $a$-summand is a polynomial in $a$ of degree at most $2m-r$.  For $1\le r\le m-1$ this degree is at most $N-2$, and the Lagrange sum vanishes; the diagonal summand is divisible by $(X-b)^m$.  For $r=0$, the reduced polynomial $(b+a)^m(b-a)^m$ has degree $N-1$ and leading coefficient $(-1)^m$, while the diagonal term is zero; the constant $-(-1)^m$ cancels this value.  Thus $P(X)^m\mid\widetilde F(X)$.  

Since $\deg P^m=M$, there is $\lambda\in\Fp$ such that
\begin{equation}\label{eq:corrected-F-factorization}
  \widetilde F(X)=\lambda P(X)^m.
\end{equation}

Fix $b\in A$.  Let $C_1=\binom{M+m}{m+1}$ and $C_2=\binom{M+m}{m+2}$.
Also set
\[
  S_1(b)=\sum_{a\in A}\frac1{P'(a)(b+a)} \qquad \text{and }\qquad S_2(b)=\sum_{a\in A}\frac1{P'(a)(b+a)^2}.
\]
Since $P(X)=X^N-\alpha^N$, we have $P'(b)=Nb^{N-1}$ and $P(-b)=-2b^N$,
and so
\begin{equation}\label{eq:odd-coset-S-values}
  2bS_1(b)P'(b)=N \qquad \text{and} \qquad (2b)^2S_2(b)P'(b)=N^2.
\end{equation}

The coefficient of $Y^m$ in equation~\eqref{eq:corrected-F-factorization} at $X=b+Y$ is
\[
  \frac{(2b)^m((2a)^M-1)}{P'(b)}=\lambda P'(b)^m.
\]
Equivalently,
\begin{equation}\label{eq:odd-coeff-m}
  \lambda P'(b)^{m+1}=(2b)^m((2a)^M-1).
\end{equation}
We compare the coefficients of $Y^{m+1}$ in \cref{eq:corrected-F-factorization}.
As in the computation of $\E^{(m)}G(b)$ in Proposition~\ref{prop:odd-nonzero-coset}, the coefficient of $Y^{m+1}$ in $\widetilde F(b+Y)$ is
\[
  \frac{(M+m)(2b)^{m-1}((2a)^M-1)}{P'(b)}+C_1(2b)^mS_1(b).
\]
We now compute the coefficient of $Y^{m+1}$ in $\lambda P(b+Y)^m$.  Since $P(X)=X^N-\alpha^N$ and $P(b)=0$, we have
\[
  P(b+Y) = (b+Y)^N-b^N = Nb^{N-1}Y+\binom N2 b^{N-2}Y^2+O(Y^3).
\]
Also $P'(b)=Nb^{N-1}$.  Thus
\[
  \frac{P(b+Y)}{P'(b)Y}
  = 1+\frac{\binom N2 b^{N-2}}{Nb^{N-1}}Y+O(Y^2)
  = 1+\frac{N-1}{2b}Y+O(Y^2)
  = 1+\frac mbY+O(Y^2).
\]
Thus the coefficient of $Y$ in $1+\frac{m}{b}Y+O(Y^2)$ to the $m$-th power is $m^2/b$.  Using equation~\eqref{eq:odd-coeff-m}, the coefficient comparison gives
\[
  \frac{(M+m)(2b)^{m-1}((2a)^M-1)}{P'(b)}+C_1(2b)^mS_1(b) = \frac{(2b)^m((2a)^M-1)}{P'(b)}\frac{m^2}{b}.
\]
Multiplying by $P'(b)/(2b)^{m-1}$ and using \cref{eq:odd-coset-S-values}, we get $(M+m)((2a)^M-1)+NC_1=2m^2((2a)^M-1)$.
Since $M+m=m(2m+1)+m=2m(m+1)$, this becomes
\begin{equation}\label{eq:odd-coeff-m+1}
  NC_1=-2m((2a)^M-1).
\end{equation}

We now compare the coefficients of $Y^{m+2}$ in $\widetilde F(b+Y)=\lambda P(b+Y)^m$.
First consider the left-hand side $\widetilde F(b+Y)$.  As in the previous
coefficient computation, the polynomial part cancels by
Lemma~\ref{lem:lagrange}, and only the diagonal correction and the rational
terms remain.
The diagonal correction is $\binom{M+m}2(2b)^{m-2}((2a)^M-1)/P'(b)$.

The non-polynomial terms come from the two largest values of the index, which are
\[
  mC_1\frac{(2b-s)^{m-1}}s = mC_1\frac{(2b)^{m-1}}s+g(s)
\quad \text{and} \quad C_2\frac{(2b-s)^m}{s^2} = C_2\frac{(2b)^m}{s^2} - mC_2\frac{(2b)^{m-1}}s +h(s)
\]
for some polynomials $g,h$ in $s$. 
The polynomial parts vanish after summing over $a\in A$, by \cref{lem:lagrange}.  Thus the surviving rational contribution is $  mC_1(2b)^{m-1}S_1(b) +  C_2(2b)^mS_2(b) - mC_2(2b)^{m-1}S_1(b)$.
Equivalently, $m(C_1-C_2)(2b)^{m- 1}S_1(b)+C_2(2b)^mS_2(b)$.
Therefore the coefficient of $Y^{m+2}$ in $\widetilde F(b+Y)$ is
\[
  \frac{\binom{M+m}{2}(2b)^{m-2}((2a)^M-1)}{P'(b)} +m(C_1-C_2)(2b)^{m-1}S_1(b)+C_2(2b)^mS_2(b).
\]

We next compute the coefficient of $Y^{m+2}$ on the right-hand side
$\lambda P(b+Y)^m$.  Since $P(b)=0$, we write
\[
  P(b+Y)=P'(b)Y \left(\frac{P(b+Y)}{P'(b)Y}\right).
\]
Then
\[
  \lambda P(b+Y)^m   =  \lambda P'(b)^mY^m \left(\frac{P(b+Y)}{P'(b)Y}\right)^m.
\]
Thus the coefficient of $Y^{m+2}$ in $\lambda P(b+Y)^m$ is $\lambda P'(b)^m[Y^2]\left(P(b+Y)/(P'(b)Y)\right)^m$.
Since $P(X)=X^N-\alpha^N$ and $P(b)=0$, we have
\[
  P(b+Y) = (b+Y)^N-b^N = Nb^{N-1}Y+\binom N2b^{N-2}Y^2+\binom N3b^{N-3}Y^3+O(Y^4).
\]
Also $P'(b)=Nb^{N-1}$.  Thus
\begin{align}\label{eq:expY}
  \frac{P(b+Y)}{P'(b)Y} =  1+\frac{N-1}{2b}Y +\frac{(N-1)(N-2)}{6b^2}Y^2 + O(Y^3).
\end{align}
Since $N=2m+1$, the coefficient of $Y$ in this expansion is $(N-1)/(2b)=m/b$.

Now write
\[
  \left(\frac{P(b+Y)}{P'(b)Y}\right)^m = 1+\frac{m^2}{b}Y+\frac{B_2}{b^2}Y^2+O(Y^3).
\]
The coefficient $B_2$ is obtained as follows.  If we raise the right-hand side of equation \cref{eq:expY} to the $m$-th power, then the coefficient of $Y^2$ is 
\[
  m\frac{(N-1)(N-2)}{6b^2}+\binom{m}{2}\left(\frac{N-1}{2b}\right)^2.
\]
Thus
\[
  \frac{B_2}{b^2} = m\frac{(N-1)(N-2)}{6b^2} + \binom m2\left(\frac{N-1}{2b}\right)^2.
\]
Equivalently,
\begin{equation}\label{eq:B2-odd}
  B_2 = m\frac{(N-1)(N-2)}6 + \binom m2\left(\frac{N-1}{2}\right)^2  = \frac{m^2(2m-1)}3+\binom m2m^2.
\end{equation}
Using \eqref{eq:odd-coeff-m}, 
we have $\lambda P'(b)^m = (2b)^m((2a)^M-1)/P'(b)$.
Therefore the coefficient of $Y^{m+2}$ on the right-hand side is
\[
  \frac{(2b)^m((2a)^M-1)}{P'(b)}\frac{B_2}{b^2} = \frac{4B_2(2b)^{m-2}((2a)^M-1)}{P'(b)}.
\]

Equating the two coefficients of $Y^{m+2}$ gives
\[
  \frac{\binom{M+m}{2}(2b)^{m-2}((2a)^M-1)}{P'(b)} + m(C_1-C_2)(2b)^{m-1}S_1(b) + C_2(2b)^mS_2(b) = \frac{4B_2(2b)^{m-2}((2a)^M-1)}{P'(b)}.
\]
Multiplying both sides by $P'(b)/(2b)^{m-2}$, we get
\[
  \binom{M+m}{2}((2a)^M-1) + m(C_1-C_2)(2b)S_1(b)P'(b) + C_2(2b)^2S_2(b)P'(b) =  4B_2((2a)^M-1).
\]
By \cref{eq:odd-coset-S-values}, we have $ (2b)S_1(b)P'(b)=N$ and $(2b)^2S_2(b)P'(b)=N^2$.
Thus
\[
  \binom{M+m}2((2a)^M-1) + mN(C_1-C_2) + N^2C_2 = 4B_2((2a)^M-1).
\]
Using $C_2=C_1\frac{M-1}{m+2}$, $N=2m+1$, and using equation~\eqref{eq:odd-coeff-m+1} to substitute $(2a)^M-1=-NC_1/(2m)$, we obtain
\[
  0=C_1N\left[-\frac1{2m}\left(\binom{M+m}{2}-4B_2\right)+m+\frac{(m+1)(M-1)}{m+2}\right].
\]
With $N=2m+1$, $M=m(2m+1)$, $M+m=2m(m+1)$, and $B_2$ as in \cref{eq:B2-odd}, the bracket equals $m(m-1)(2m-1)/(6(m+2))$.
Thus
\[
  0=C_1N\frac{m(m-1)(2m-1)}{6(m+2)}.
\]
For $m\ge2$, each factor is nonzero in $\Fp$. Indeed, the binomial coefficient $C_1$ is nonzero since $M+m<p$, and the remaining integer factors have absolute value less than $p$.  This gives a contradiction.
\end{proof}

\section[The zero-even case]{\texorpdfstring{The zero-even case: $0\in A$ and $|A|$ is even}{The zero-even case}}\label{sec:5}

In this section, we handle the case when the set contains zero and has even cardinality. 


\begin{proposition}\label{prop:zero-even-root-coset}
Let $d\ge2$, and let $p\equiv1\pmod d$ be a prime with $p\ge4d+1$.
Assume $0\in A$, $A\widehat{+}A=S_d$, and $N=|A|$ is even. Put
$B=A\setminus\{0\}$ with $|B|=N-1$.
Then there exists $\gamma\in S_d$ and a multiplicative subgroup
$H\le S_d$ of order $N-1$ such that $B=\gamma H$.
\end{proposition}

\begin{proof}
Since $N$ is even, write $N=2m+2$. 
By \cref{prop:restricted-sum-sidon-prime}, we have
$M=\binom N2$, $A$ is Sidon, and $2A\cap S_d=\varnothing$.
Let
\[
P(X)=\prod_{a\in A}(X-a).
\]
Consider the even auxiliary polynomial
\[
f(X)= -(-1)^{m+1} + \sum_{a\in A} \frac{(X+a)^{M+m}(X-a)^{m+1}}{P'(a)}.
\]
The degree and multiplicity calculation for $f$ was carried out in the proof of Proposition~\ref{prop:restricted-sum-sidon-prime}. Indeed, one has $\deg f\le M$ and $f$ has a zero of multiplicity at least $m+1$ at every point $a\in A$. Since $N(m+1)>M=\binom N2$, we must have $f\equiv0$.

For $a\in B=A\setminus\{0\}$, we may repeat the derivative comparison
from the even case in the proof of Proposition~\ref{prop:restricted-sum-sidon-prime}.
The rational sums appearing there are
\[
  S_1(a)=\sum_{u\in A}\frac{1}{P'(u)(a+u)} \qquad\text{and}\qquad S_2(a)=\sum_{u\in A}\frac{1}{P'(u)(a+u)^2}.
\]
Thus we only need to check that $a+u\ne0$ for every $u\in A$.
If $u\ne a$, then $a+u$ is a restricted sum, so $a+u\in A\widehat{+}A=S_d\subset\mathbb F_p^*$.
This includes the case $u=0$, where $a+u=a\in B\subset S_d$.
If $u=a$, then $a+u=2a\ne0$, since $a\ne0$.
Thus all denominators in $S_1(a)$ and $S_2(a)$ are nonzero, and
the same derivative comparison gives
\begin{equation}
\label{eq:zero-even-local-differential}
  2aQ_0'(a)=(N+1)Q_0(a),
  \qquad \text{where} \quad Q_0(X)=P(-X).
\end{equation}

Define $H_0(X):=2XQ_0'(X)-(N+1)Q_0(X)$.
By equation~\eqref{eq:zero-even-local-differential}, $H_0$ vanishes at every
$a\in B$. Since $0\in A$, we have $P(0)=0$, so $Q_0(0)=0$, and
$H_0(0)=0$. Thus $H_0$ vanishes at every element of $A$.

The leading coefficient of $H_0$ is $2N-(N+1)=N-1$. 
Since $N<p$, we have $N-1\neq0$ in $\mathbb F_p$. Thus $H_0(X)=(N-1)P(X)$.
That is, 
\begin{equation}\label{eq:H0-equals-P}
2XQ_0'(X)-(N+1)Q_0(X)=(N-1)P(X).
\end{equation}

Now write $P(X)=P_{\rm ev}(X)+P_{\rm odd}(X)$, where $P_{\rm ev}$ and $P_{\rm odd}$ are the even and odd parts of $P$.
Since $N$ is even, $Q_0(X)=P(-X)=P_{\rm ev}(X)-P_{\rm odd}(X)$.
Comparing even and odd parts in equation~\eqref{eq:H0-equals-P}, we get
$XP_{\rm ev}'(X)=NP_{\rm ev}(X)$ and $XP_{\rm odd}'(X)=P_{\rm odd}(X)$. 
Since $\deg P=N<p$, these differential equations imply $P_{\rm ev}(X)=X^N$ and $P_{\rm odd}(X)=\beta X$ for some $\beta\in\mathbb F_p$. Thus
\[
P(X)=X^N+\beta X=X(X^{N-1}+\beta).
\]
Since $A$ has $N$ distinct elements, $\beta\neq0$. Thus $B=\{x\in\mathbb F_p^*:x^{N-1}=-\beta\}$.
Thus $B$ is a multiplicative coset of the subgroup of $(N-1)$-st roots
of unity, and so there exists $\gamma\in\mathbb F_p^*$ and a multiplicative subgroup $H\leq\mathbb F_p^*$ of order $N-1$ such that $B=\gamma H$.

Since $0\in A$, we have $B\subset S_d$.  Thus $\gamma\in S_d$, and therefore $H=\gamma^{-1}B\le S_d$.
\end{proof}


By \cref{prop:zero-even-root-coset}, since $B=\gamma H$ with $\gamma\in S_d$, multiplying the identity $A\widehat{+}A=S_d$ by $\gamma^{-1}\in S_d$ preserves $S_d$.  Thus, after this normalization, we may assume
\[
A=\{0\}\sqcup H,\qquad H\le S_d,\qquad |H|=n=2m+1.
\]
Then
\[
S_d=H\sqcup(H\widehat{+}H),\qquad M=\frac{n(n+1)}2=n(m+1).
\]
The cases $m=0$ and $m=1$ correspond to $M=1$ and $M=6$,
respectively; these are among the small cases discussed in
\cref{rem:small-exceptions}.  We now assume $m\ge2$.

\begin{proposition}\label{prop:zero-even}
Let $d\ge2$, and let $p\equiv1\pmod d$ be a prime with $p\ge4d+1$.  There is no set $A\subset\mathbb F_p$ such that $A\widehat{+}A=S_d$, $0\in A$, and $|A|=2m+2$ for some $m\ge2$.
\end{proposition}


\begin{proof}
By the preceding normalization, we may assume that
$A=\{0\}\sqcup H$, where $H\le S_d$ has order $n=2m+1$ and $S_d=H\sqcup(H\widehat{+}H)$.

Let
\[
  Q(X)=X^n-1=\prod_{h\in H}(X-h) \quad \text{and} \quad  c_h=\frac1{Q'(h)}.
\]
Define
\begin{equation}\label{eq:zero-even-Phi}
  \Psi(X)=-(-1)^{m+1} + \sum_{h\in H}c_h(X+h)^{M+m-1}(X-h)^{m+1}.
\end{equation}

As in the proofs of the preceding propositions, one has $\deg\Psi\le M$.
In addition, one can also verify that $ X^nQ(X)^m\mid\Psi(X)$. 
We omit the repeated details. 


Since $X^n$ and $Q(X)$ are coprime and $\deg(X^nQ(X)^m)=n+mn=n(m+1)=M$,
we have
\[
  \Psi(X)=\lambda X^nQ(X)^m
\]
for some $\lambda\in\Fp$.

For $j=1,2,3$, put
\[
  T_j=\sum_{h\in H}\frac1{Q'(h)(1+h)^j}.
\]
Since $n$ is odd,
\[
  \sum_{h\in H}\frac1{Q'(h)(X+h)}=-\frac1{Q(-X)}=\frac1{X^n+1}.
\]
Evaluating at $X=1$ and differentiating gives $T_1=1/2$, $T_2=n/4$, and $T_3=n/8$.
Also, let $C_0=\binom{M+m-1}m, C_1=\binom{M+m-1}{m+1}$, and $C_2=\binom{M+m-1}{m+2}$.
Since $H\le S_d$ and $2H\cap S_d=\varnothing$, we have $2\notin S_d$, so $2^M\ne1$.

We compare the coefficients of $Y^m$ in $\Psi(1+Y)=\lambda(1+Y)^nQ(1+Y)^m$.
First consider the left-hand side.  For $u\in H$, put $s=1+u$.  
As in the proof of the previous propositions, 
Thus
\[
  \E^{(m)}\Psi(1) = C_0 2^{m+1}\sum_{u\in H}\frac{c_u}{1+u} = C_0 2^{m+1}T_1.
\]
Since $T_1=1/2$, $ \E^{(m)}\Psi(1)=2^mC_0$.
On the right-hand side, since $Q(1+Y)=(1+Y)^n-1=nY+O(Y^2)$, we have $ (1+Y)^nQ(1+Y)^m = (1+O(Y))(nY+O(Y^2))^m$.
Thus the coefficient of $Y^m$ is $n^m$, and the coefficient of $Y^m$ in $\lambda(1+Y)^nQ(1+Y)^m$ is $\lambda n^m$.  Comparing the two sides gives
\begin{equation}\label{eq:zero-even-coeff-m}
  \lambda n^m=2^mC_0.
\end{equation}

We next compare the coefficients of $Y^{m+1}$ in $\Psi(1+Y)=\lambda(1+Y)^nQ(1+Y)^m$.
Again, we skip the repeated details, and one obtains that
using $T_1=1/2,T_2=n/4$, the total coefficient of $Y^{m+1}$ on the
left-hand side is
\[
  2^{m-1}\left( \frac{2^M-1}{n} +(m+1)(C_0-C_1) +nC_1\right).
\]

We now compute the right-hand side.  Write
\[
  U(Y)=\frac{Q(1+Y)}{nY} =\frac{(1+Y)^n-1}{nY} =1+mY+O(Y^2),
\]
since $n=2m+1$.  Then $ Q(1+Y)^m=(nY)^mU(Y)^m$.
Thus
\[
  \lambda(1+Y)^nQ(1+Y)^m = \lambda n^mY^m(1+Y)^nU(Y)^m.
\]
Thus the coefficient of $Y^{m+1}$ on the right-hand side is
\[
  \lambda n^m \cdot [Y]\bigl((1+Y)^nU(Y)^m\bigr).
\]
Since $(1+Y)^n=1+nY+O(Y^2)$ and $U(Y)^m=(1+mY+O(Y^2))^m=1+m^2Y+O(Y^2)$, the coefficient of $Y$ in $(1+Y)^nU(Y)^m$ is $n+m^2=(m+1)^2$.
Using \eqref{eq:zero-even-coeff-m}, namely $\lambda n^m=2^mC_0$, the right-hand side coefficient is $2^mC_0(m+1)^2$.

Comparing the two sides gives
\[
  \frac{2^M-1}{n} +(m+1)(C_0-C_1) +nC_1 = 2C_0(m+1)^2.
\]
Since $C_1=C_0(M-1)/(m+1)$, this simplifies to 
\begin{equation}\label{eq:zero-even-coeff-m+1}
  2^M-1=n\frac{M+m}{m+1}C_0.
\end{equation}

We now compare the coefficients of $Y^{m+2}$ in $\Psi(1+Y)=\lambda(1+Y)^nQ(1+Y)^m$.
We first consider the left-hand side.  
For $u\in H$, put
$s=1+u$.
Using $T_1=1/2,T_2=n/4$, and $T_3=n/8$, the coefficient of $Y^{m+2}$ 
the total coefficient of $Y^{m+2}$ on the
left-hand side is
\[
  2^{m-2}\left(\frac{(M+m-1)(2^M-1)}n +\frac{m(m+1)}2C_0 + (m+1)^2C_1 -\frac{m(3m+1)}2C_2 \right).
\]
We omit the repeated details again. 

We now compute the right-hand side.  Recall that $  Q(1+Y)=nY\,U(Y)$, where $U(Y)=((1+Y)^n-1)/(nY)$.
Thus $\lambda(1+Y)^nQ(1+Y)^m = \lambda n^mY^m(1+Y)^nU(Y)^m$.
Therefore the coefficient of $Y^{m+2}$ on the right-hand side is $\lambda n^m \cdot [Y^2]\bigl((1+Y)^nU(Y)^m\bigr)$.
Write
\[
  U(Y) = 1+mY+\frac{m(2m-1)}3Y^2+O(Y^3),
\]
since $n=2m+1$.  Thus
\[
  U(Y)^m = 1+m^2Y+ \left(\frac{m^2(2m-1)}3+\binom m2m^2  \right)Y^2+O(Y^3).
\]
Since
\[
  (1+Y)^n = 1+nY+\binom n2Y^2+O(Y^3) = 1+nY+nmY^2+O(Y^3),
\]
the coefficient of $Y^2$ in $(1+Y)^nU(Y)^m$ is
\begin{equation}\label{eq:D2}
  D_2 := nm+nm^2+\frac{m^2(2m-1)}3+\binom m2m^2.
\end{equation}
Thus, using \cref{eq:zero-even-coeff-m}, the right-hand side coefficient is $2^mC_0D_2$.

Comparing the two sides gives
\begin{equation}\label{eq:zero-even-coeff-m+2}
  \frac{(M+m-1)(2^M-1)}n +\frac{m(m+1)}2C_0 + (m+1)^2C_1 -\frac{m(3m+1)}2C_2 = 4C_0D_2.
\end{equation}

Now substitute \cref{eq:zero-even-coeff-m+1},
\[
  C_1=C_0\frac{M-1}{m+1}, \qquad C_2=C_0\frac{(M-1)(M-2)}{(m+1)(m+2)},
\]
\cref{eq:D2}, and $n=2m+1$, $M=(2m+1)(m+1)$ into \cref{eq:zero-even-coeff-m+2}.  After factoring out $C_0$, we have
\[
\begin{aligned}
&\frac{(M+m-1)(M+m)}{m+1} + \frac{m(m+1)}2 + (m+1)(M-1) \\
&\quad -\frac{m(3m+1)(M-1)(M-2)}{2(m+1)(m+2)}
-4\left(nm+nm^2+\frac{m^2(2m-1)}3+\binom m2m^2\right) \\
&= -\frac{m(m-1)(2m+3)(3m+1)(4m+7)}{3(m+2)}.
\end{aligned}
\]
Thus
\[
  0=-C_0\frac{m(m-1)(2m+3)(3m+1)(4m+7)}{3(m+2)}.
\]
For $m\ge2$, every factor in the above equation is nonzero in $\Fp$, including $C_0$ since $M+m-1<p$; the remaining integer factors are all smaller than $p$.  This contradiction proves the proposition.
\end{proof}

\section{Proof of \cref{thm:main} and exceptional cases}\label{sec:6}

We now assemble the case analysis.  \cref{prop:restricted-sum-sidon-prime}
first reduces any decomposition to the critical case $M=\binom{|A|}{2}$, with $0\in A$ or $|A|$ odd.  
The subsequent case analyses, namely \cref{prop:zero-odd}, \cref{prop:odd-coset-contradiction}, \cref{prop:zero-even}, eliminate all possibilities except the small triangular values $M=1,3,6$.  These are all below the stated range $M\ge7$.

\begin{proof}[Proof of \cref{thm:main}]
Suppose, for contradiction, that $A\rsum A=S_d$. 
Since $p\ge 7d+1$, we have $M\ge 7$.

By Proposition~\ref{prop:restricted-sum-sidon-prime}, $A$ is Sidon, $2A\cap S_d=\varnothing$, $M=\binom{|A|}{2}$, and either $0\in A$ or $|A|$ is odd.

If $0\in A$ and $|A|$ is odd, this contradicts
\cref{prop:zero-odd}.

Next suppose $0\notin A$ and $|A|$ is odd.  Write $|A|=2m+1$.
By \cref{prop:odd-coset-contradiction}, we must have
$m\le1$.  Since $|A|\ge2$, the only remaining possibility is
$m=1$.  Then $|A|=3$, so $M=\binom{3}{2}=3$, contradicting $M\ge7$.

It remains to consider the case $0\in A$ and $|A|$ is even.  Write
$A=\{0\}\sqcup B$ with $|B|=2m+1$.
By Proposition~\ref{prop:zero-even}, we must have $m\le1$.
If $m=0$, then $|A|=2$, and so $M=\binom{2}{2}=1$, contradicting $M\ge7$.  If $m=1$, then $|A|=4$, and so $M=\binom{4}{2}=6$, again contradicting $M\ge7$.

All cases lead to a contradiction.  Therefore no such set $A$ exists.
\end{proof}


\begin{remark}\label{rem:small-exceptions}
The condition $p\ge7d+1$ is equivalent to $M=|S_d|\ge7$.  Below this range, $M\le6$, and the small cases can be listed explicitly.

\begin{itemize}
    \item If $p=d+1$ is prime, then $S_d=\{1\}$.  In this case $A=\{x,1-x\}$ with $x\ne \frac12$ satisfies $A\rsum A=S_d$. Conversely, suppose that $A\rsum A=\{1\}$. Fix $a\in A$. Then every $b\in A\setminus\{a\}$ must satisfy $b=1-a$, and so $|A|\le2$. Since the restricted sumset is nonempty, $|A|=2$, so $A=\{x,1-x\}$.

    \smallskip
    \item If $p=2d+1$ is prime, then $|S_d|=2$, and no decomposition exists.
Indeed, assume that $A\rsum A=S_d$.  Fix $a\in A$.  Then the map $A\setminus\{a\}\rightarrow S_d,\;  b\mapsto a+b$
is injective, and so $|A|-1\le |S_d|=2$.  Thus $|A|\le3$.

If $|A|=2$, then $A\rsum A$ consists of a single element, so it cannot be $S_d$.  
If $|A|=3$, say $A=\{a_1,a_2,a_3\}$, then the three
restricted sums $a_1+a_2, a_1+a_3, a_2+a_3$ are distinct.  Thus $|A\rsum A|=3$, again impossible
since $|S_d|=2$.  Therefore no decomposition exists.

\smallskip
\item If $p=3d+1$ is prime, then $|S_d|=3$, and the unique
decomposition is $A=-S_d$ with $|A|=3$.
Indeed, first, we show that $A=-S_d$ gives a decomposition when $|A|=3$.  
Since $|S_d|=3$, we may write
$S_d=\{1,\omega,\omega^2\}$, where $\omega$ has order $3$.  Then $ 1+\omega+\omega^2=0$.
Thus $(-1)+(-\omega)=\omega^2$, $(-1)+(-\omega^2)=\omega$, and $(-\omega)+(-\omega^2)=1$.
Thus $(-S_d)\rsum(-S_d)=S_d$.

Now assume that $A\rsum A=S_d$. For each fixed $a\in A$, again, the map $A\setminus\{a\}\rightarrow S_d,\;  b\mapsto a+b$
is injective, and so $|A|\le4$.  The cases $|A|\le2$ are impossible since $|S_d|=3$.

If $|A|=4$, assume that $A\rsum A=S_d$. Then for every $a\in A$ the three sums $a+b$ with $b\ne a$ are all elements of $S_d$. Let $T=\sum_{a\in A}a$.
Since the sum of the elements of $S_d$ is zero, we obtain
$\sum_{b\in A, b\ne a}(a+b)=0$.
The left-hand side is
\[
        3a+\sum_{b\in A, b\ne a}b =3a+(T-a) =2a+T.
\]
Thus $2a+T=0$ for every $a\in A$, which is impossible since the elements of $A$ are distinct, a contradiction.

Assume $|A|=3$.  If $T=\sum_{a\in A}a$, then the restricted sums are $T-a$ for every $a\in A$.
Thus $T-A=S_d$, and so $A=T-S_d$.
It follows that $A\rsum A=2T-(S_d\rsum S_d)$.
Note that we have $S_d\rsum S_d=-S_d$.
Indeed, if $S_d=\{h_1,h_2,h_3\}$, then $h_1+h_2+h_3=0$, so the sum of any two distinct elements of $S_d$ is the negative of the third.  Then $2T+S_d=A\rsum A=S_d$.
Then
\[0=\sum_{s\in S_d}s=\sum_{s\in S_d}(2T+s)=6T+\sum_{s\in S_d}s=6T.\]
As $p=3d+1$ and $d\ge2$, we have $p\ne2,3$, so $T=0$.  Thus $A=-S_d$.

\smallskip
\item  There are no examples with $p=4d+1$ or $p=5d+1$.  
Indeed, in those cases $M=|S_d|$ is respectively $4$ or $5$. If $A\rsum A=S_d$, by \cref{prop:restricted-sum-sidon-prime}, $|A\rsum A|=M=\binom{|A|}{2}$,  but neither $4$ nor $5$ is of the form $\binom n2$ for some positive integer $n$.  This is a
contradiction.

\smallskip
\item Finally, if $p=6d+1$ is prime, then $|S_d|=6$.  In this case decompositions do occur.  
Let $H$ be the unique subgroup of $S_d$ of order $3$. Thus $-1 \notin H$.
Since $S_d$ is cyclic of order $6$, we have $S_d=H\sqcup(-H)$.
Writing $H=\{1,\omega,\omega^2\}$ with $1+\omega+\omega^2=0$, we get $H\rsum H=-H$.
Thus
\[
        (\{0\}\sqcup H)\rsum(\{0\}\sqcup H) = H\sqcup(H\rsum H) = H\sqcup(-H) = S_d.
\]
The same argument with $-H$ in place of $H$ gives the second example $A=\{0\}\sqcup(-H)$.

Conversely, if $p=6d+1$ and $A\rsum A=S_d$, then
\cref{prop:restricted-sum-sidon-prime} gives $M=\binom{|A|}{2}$.
Since $M=6$, we have $|A|=4$.  
Then \cref{prop:restricted-sum-sidon-prime} gives $0\in A$.
Writing $A=\{0\}\sqcup B$ and $|B|=3$, \cref{prop:zero-even-root-coset} gives $B=\gamma H$, where $H\le S_d$ has order $3$ and $\gamma\in S_d$.  Since $S_d/H$ has order $2$, the coset $\gamma H$ is either $H$ or $-H$. Thus the only decompositions for $p=6d+1$ are $A=\{0\}\sqcup H$ and $A=\{0\}\sqcup(-H)$.
\end{itemize}
\end{remark}


\section*{Acknowledgments}
The authors thank Ilya Shkredov for helpful discussions.  
S. Yoo was supported by the Institute for Basic Science (IBS-R029-C1).

\bibliographystyle{abbrv}
\bibliography{references}

\end{document}